\documentclass[12pt,reqno]{amsart}

\usepackage{verbatim}

\newtheorem{thm}{Theorem}

\theoremstyle{definition}

\providecommand{\RR}{\mathbb{R}}

\providecommand{\NN}{\mathbb{N}}

\def\mn{\mathcal{M}_n}

\def\sn{\mathcal{S}_n}

\def\M{\mathcal{M}}

\def\rr{\mathrm{real}}
\def\cE{\mathcal{E}}
\def\chrisA{\mathcal{A}}
\begin{document}

\title[A Real Nullstellensatz for free modules]{A Real Nullstellensatz for free modules}

\author{J. Cimpri\v c}



\date{\today}

\address{Jaka Cimpri\v c, University of Ljubljana, Faculty of Math. and Phys.,
Dept. of Math., Jadranska 19, SI-1000 Ljubljana, Slovenija. 
E-mail: cimpric@fmf.uni-lj.si. www page: http://www.fmf.uni-lj.si/ $\!\!\sim$cimpric.}

\begin{abstract}
Let $\chrisA$ be the algebra of all $n \times n$ matrices with entries from $\RR[x_1,\ldots,x_d]$
and let $G_1,\ldots,G_m,F \in \chrisA$. We will show that $F(a)v=0$
for every $a \in \RR^d$ and $v \in \RR^n$ such that $G_i(a)v=0$ for all $i$
if and only if $F$ belongs to the smallest real left ideal of $\chrisA$ which contains $G_1,\ldots,G_m$.
Here a left ideal $J$ of $\chrisA$ is real if for every
$H_1,\ldots,H_k \in \chrisA$ such that $H_1^T H_1+\ldots+H_k^T H_k \in J+J^T$
we have that $H_1,\ldots,H_k \in J$.
We call this result the one-sided Real Nullstellensatz for matrix polynomials.
We first prove by induction on $n$ that it holds when $G_1,\ldots,G_m,F$ have zeros everywhere except in the first row.
This auxiliary result can be formulated as a Real Nullstellensatz for the free module 
$\RR[x_1,\ldots,x_d]^n$. 
\end{abstract}

\maketitle

\thispagestyle{empty}

\section{Introduction}

For given polynomials $g_1,\ldots,g_m \in \RR[x]$, where $x=(x_1,\ldots,x_d)$, the Positivstellensatz \cite{krivine,stengle} gives an algebraic characterization
of the polynomials $f \in \RR[x]$ which are nonnegative on the set $\{a\in \RR^d \mid g_1(a) \ge 0,\ldots,g_m(a)\ge 0\}$. Its
special case, the Real Nullstellensatz \cite{dub, ris,efr}, gives an algebraic characterization of the 
polynomials $f \in \RR[x]$ which vanish on the set $\{a\in \RR^d \mid g_1(a) = 0,\ldots,g_m(a)= 0\}$.

Noncommutative Real Algebraic Geometry strives to generalize the Positivstellensatz and the Real Nullstellensatz
to other $\ast$-algebras. The first nontrivial example is the $\ast$-algebra $\mn(\RR[x])$ 
of all real matrix polynomials of size $n$. The one-sided part of  the theory considers the following problems where
$\sn(\RR[x])=\{H \in \M_n(\RR[x]) \mid H^T=H\}$.
\begin{enumerate}
\item (One-sided Positivstellensatz) Given $G_1,\ldots,G_m \in \sn(\RR[x])$, characterize all $F \in \sn(\RR[x])$
such that $\langle F(a)v,v \rangle \ge 0$ for all $a \in \RR^d$ and all $v \in \RR¢^n$ for which
$\langle G_i(a) v,v \rangle \ge 0$ for all $i$.
\item (One-sided Real Nullstellensatz) Given $G_1,\ldots,G_m \in \mn(\RR[x])$, characterize all $F \in \mn(\RR[x])$
such that $F(a)v=0$ for all $a \in \RR^d$ and all $v \in \RR¢^n$ satisfying $G_1(a)v=\ldots=G_m(a)v=0$.
\end{enumerate}
The aim of this paper is to solve problem (2). The special case $d=1$ has already been solved in \cite[Section 6]{chmn}.
Problem (1) is widely open even when $G_i$ and $F$ are constant and $m \ge 3$; see \cite{slemma}. If $G_1,\ldots,G_m$ belong to the center of $\mn(\RR[x])$,
then problem (1) coincides with problem (3) below. A variant of problem (1) is considered in \cite[Theorem 2.1]{c2}.

The two-sided part of the theory considers the following problems.
\begin{enumerate}
\setcounter{enumi}{2}
\item (Two-sided Positivstellensatz) Given $G_1,\ldots,G_m \in \sn(\RR[x])$, characterize all $F \in \sn(\RR[x])$
such that $F(a)$ is positive semidefinite for all $a \in \RR^d$ for which all $G_i(a)$ are positive semidefinite.
\item (Two-sided Real Nullstellensatz) Given $G_1,\ldots,G_m \in \mn(\RR[x])$, characterize all $F \in \mn(\RR[x])$
such that $F(a)=0$ for all $a \in \RR^d$ satisfying $G_1(a)=\ldots=G_m(a)=0$.
\end{enumerate}
Problems (3) and (4) have already been solved in \cite[Theorems 14 and 18]{c1} by using the ideas from \cite[Section 4]{sch2}.
In special cases, however, stronger results exist; see \cite{sh1,sh2,sh3,zalar2} if $G_1(x)=R^2-\sum_{i=1}^d x_i^2$ for some $R$,
\cite{j,d,ds,ss} if $d=1$ and \cite{klep1,klep2,zalar1} if $G_i$ and $F$ are linear.

Problems (1)-(4) also make sense for free $\ast$-algebras if we replace evaluations in real points
with finite-dimensional $\ast$-representations, i.e. with evaluations in $d$-tuples of same-size real matrices. 
In this context, problem (2) was solved in \cite{chmn}. For additional information see \cite{chkmn} and \cite{nelson}. 
Special cases of problems (1) and (3) and (4) were solved in \cite{hereditary,hmp} and \cite{ncsos, isometries} and \cite{chmn2} respectively.

Problem (3) is also solved for the Toeplitz algebra \cite{ncfejer}
and partial results exist for Weyl algebras \cite{weyl1, weyl2,nahas}, enveloping algebras of finite-dimensional Lie-algebras \cite{env}
and the quantum plane \cite{qua}. Here we replace evaluations in real points  with ``well-behaved" $\ast$-representations. 
On the other hand, we are not aware of any results related to problems (1) and (2) for these algebras, but \cite{shankar} raises hopes.

The plan of our paper is as follows. Let $R$ be a  finitely generated commutative unital $\RR$-algebra. 
We will first recall the Real Nullstellensatz for $R$ (see Theorem \ref{thm1}). We will then formulate and prove
a Real Nullstellensatz for the free module $R^n$ (see Theorem \ref{thm2}). The proof is by induction on $n$ and uses Theorem \ref{thm1} at each step.
Theorem \ref{thm3} is the solution of problem (2). It is just a reformulation of Theorem \ref{thm2}.

Let us write $V_\RR(R)$ for the set of all $\RR$-algebra homomorphisms from $R$ to $\RR$. 
By the Artin-Lang homomorphism theorem, $V_\RR(R)$ is  nonempty if and only if $-1 \not\in \sum R^2$.
If $R=\RR[x_1,\ldots,x_d]$ for some $d$ then $V_\RR(R)$ can be identified with $\RR^d$.

For every ideal $I$ of $R$ we define its \textbf{real zero set} by
$$Z_\RR(I):=\{ \phi \in V_\RR(R) \mid \phi(g)=0 \text{ for all } g \in I\}$$
and its \textbf{real saturation} by
$$\sqrt[\cE]{I}:=\{ f \in R \mid \phi(f)=0 \text{ for all } \phi \in Z_\RR(I)\}.$$
Recall that an ideal $J$ of $R$ is \textbf{real} if
for every $r_1,\ldots,r_K \in R$ such that $\sum_{i=1}^K r_i^2 \in J$ we have that
$r_1, \ldots, r_K \in J$. The \textbf{real radical} $\sqrt[\rr]{I}$ of an ideal $I$ of $R$ is 
by definition the smallest real ideal which contains $I$. It is easy to show that
\begin{equation}
\label{eq1}
\sqrt[\rr]{I} = \big\{ f \in R \mid \begin{array}{c} f^{2k}+\sum_{i=1}^L s_i^2 \in I \text{ for some } k,L \in \NN\\  \text{ and some } s_1,\ldots,s_L \in R \end{array}\big\}.
\end{equation}

\begin{thm}[Real Nullstellensatz] 
\label{thm1}
For every ideal $I$ of $R$,
$$\sqrt[\cE]{I}=\sqrt[\rr]{I}.$$
\end{thm}

\section{A Real Nullstellensatz for free modules}

Let $n$ be a fixed integer and let $M$ be the free $R$-module of rank $n$, i.e. $M=R^n$. 
We would like to generalize Theorem \ref{thm1} to submodules of $M$. 
Let $V_\RR(M)$ be the set of all pairs $(\phi,\mathbf{u})$ where $\phi \in V_\RR(R)$ and $\mathbf{u} \in \RR^n$.
For every submodule $N$ of $M$ we define its \textbf{real zero set}
$$Z_\RR(N):=\{(\phi,\mathbf{u}) \in V_\RR(M) \mid \langle \phi^n(\mathbf{g}),\mathbf{u} \rangle =0 \text{ for every } \mathbf{g} \in N\}$$
and its \textbf{real saturation} 
$$\sqrt[\cE]{N}:=\{\mathbf{f} \in M \mid \langle \phi^n(\mathbf{f}),\mathbf{u} \rangle =0  \text{ for all } (\phi,\mathbf{u}) \in Z_\RR(N)\}.$$
Here $\phi^n(\mathbf{h})$ means that we apply $\phi$ componentwise to $\mathbf{h}$ and $\langle \cdot,\cdot \rangle$ is the standard inner product on $\RR^n$,
hence $\langle \phi^n(\mathbf{h}),\mathbf{u} \rangle=\sum_{i=1}^n \phi(h_i)u_i$.

We say that a submodule $N$ of $M$ is \textbf{real} if for every $\mathbf{m}_1,\ldots,\mathbf{m}_K \in M$
such that $\sum_{i=1}^K \mathbf{m}_i \otimes \mathbf{m}_i \in M \otimes N+N \otimes M$ we have that $\mathbf{m}_1,\ldots,\mathbf{m}_K \in N$.
(Here we consider $M \otimes N$ and $N \otimes M$ as submodules of $M \otimes M$ and all tensor products
are over $R$.) For every submodule $N$ of $M$ there exists the smallest real submodule of $M$ which contains it.
We will denote it by $\sqrt[\rr]{N}$ and call it the \textbf{real radical} of $N$.

An analogue of formula \eqref{eq1} exists but it is more complicated. We have to introduce first an
\textbf{auxiliary radical}
\begin{equation}
\sqrt[\alpha]{N} = \big\{\mathbf{f} \in M \mid
\begin{array}{c}
 \mathbf{f} \otimes \mathbf{f}+\sum_{i=1}^L \mathbf{s}_i \otimes \mathbf{s}_i \in M \otimes N+N \otimes M
\\
\text{ for some } L \in \NN \text{ and } \mathbf{s}_1,\ldots,\mathbf{s}_L \in M
\end{array}
\big\}.
\end{equation}
Clearly, the set $\sqrt[\alpha]{N}$ is a submodule of $M$ and
\begin{equation}
\label{eq3}
\sqrt[\rr]{N}=N \cup \sqrt[\alpha]{N} \cup \sqrt[\alpha]{\sqrt[\alpha]{N}} \cup \ldots.
\end{equation}

Here is the main result of this paper.

\begin{thm}
\label{thm2} For every submodule $N$ of $R^n$
$$\sqrt[\cE]{N}=\sqrt[\rr]{N}.$$
\end{thm}

\begin{proof}
To prove the inclusion $\sqrt[\rr]{N} \subseteq \sqrt[\cE]{N}$ it suffices to show that $\sqrt[\cE]{N}$ is a real submodule of $R^n$
containing $N$. Suppose that $$\sum_i \mathbf{s}_i \otimes \mathbf{s}_i \in M \otimes \sqrt[\cE]{N}+\sqrt[\cE]{N} \otimes M.$$
If $\langle \phi^n(N),\mathbf{u}\rangle=0$ for some $(\phi,\mathbf{u}) \in V_\RR(R)$,
then also $\sum_i \langle \phi^n(\mathbf{s}_i),\mathbf{u} \rangle^2=
\langle (\phi^n \otimes \phi^n)(\sum_i \mathbf{s}_i \otimes \mathbf{s}_i)), \mathbf{u} \otimes \mathbf{u} \rangle \in
\langle \phi^n(M) \otimes \phi^n(N)+\phi^n(N) \otimes \phi^n(M),\mathbf{u} \otimes \mathbf{u} \rangle
= \langle \phi^n(M), \mathbf{u} \rangle \langle \phi^n(N), \mathbf{u} \rangle +
\langle \phi^n(N), \mathbf{u} \rangle \langle \phi^n(M), \mathbf{u} \rangle =0.$
It follows that $\langle \phi^n(\mathbf{s}_i),\mathbf{u} \rangle=0$ for every $i$, and so, 
$\mathbf{s}_i \in \sqrt[\cE]{N}$ for every $i$.

We will prove the opposite inclusion, $\sqrt[\cE]{N} \subseteq \sqrt[\rr]{N}$, by induction on $n$. The case $n=1$ clearly follows from Theorem \ref{thm1}. 
Suppose now that the inclusion holds for every submodule of $R^{n-1}$ and take any submodule $N$ of $M=R^n$.

Since $R$ is Noetherian, so is $M$. In particular, $N$ is finitely generated, i.e. 
there exist $\mathbf{g}_1,\ldots,\mathbf{g}_m \in M$ such that $$N=\sum_{i=1}^m R \, \mathbf{g}_i.$$
For each $i=1,\ldots,m$ we write
$$\mathbf{g}_i =\sum_{j=1}^n g_{ij} \mathbf{e}_j$$
where $\mathbf{e}_1,\ldots,\mathbf{e}_n$ is the standard basis of $M$. Now pick any
$$\mathbf{f}=\sum_{j=1}^n f_j \mathbf{e}_j \in \sqrt[\cE]{N}.$$ 
The proof that $\mathbf{f} \in \sqrt[\rr]{N}$ will be split into two claims.

\subsection*{Claim 1} $g_{kl} \mathbf{f} \in \sqrt[\rr]{N}$ for every $k=1,\ldots,m$ and $l=1,\ldots,n$.

\medskip

For simplicity we will prove Claim 1 only for $l=1$ but the same argument also works for other $l$.
Pick $k=1,\ldots,m$ and write
$$\mathbf{g}'_i=\sum_{j=2}^n (g_{k1} g_{ij}- g_{kj}g_{i1})\mathbf{e}'_{j-1}$$
for all $i=1,\ldots,m$, where $\mathbf{e}'_1,\ldots,\mathbf{e}'_{n-1}$ is the standard basis of $R^{n-1}$.
Let $N'$ be the submodule of $R^{n-1}$ generated by $\mathbf{g}'_1,\ldots, \mathbf{g}'_m$. 
Let us verify that the element $$\mathbf{f}' = \sum_{j=2}^n (g_{k1} f_j-g_{kj} f_1)\mathbf{e}'_{j-1}$$
belongs to $\sqrt[\rr]{N'}$. 

By the induction hypothesis, it suffices to verify that $\mathbf{f}' \in \sqrt[\cE]{N'}$. 
Suppose that for some $\phi \in V_\RR(R)$ and some $\mathbf{v} \in R^{n-1}$, we have that 
$$\sum_{j=2}^n \phi(g_{k1} g_{ij}- g_{kj}g_{i1})v_{j-1}=0 \text{ for } i=1,\ldots,m.$$
Then $\phi(g_{i1}) u_1+\ldots+\phi(g_{in})u_n=0$ for $u_1=-\sum_{j=2}^n \phi(g_{kj}) v_{j-1}$.
$u_2=\phi(g_{k1})v_1,\ldots,u_n=\phi(g_{k1})v_{n-1}$ and all $i=1,\ldots,m$.
Since $\mathbf{f} \in \sqrt[\cE]{N}$, it follows that $\phi(f_1) u_1+\ldots+\phi(f_n)u_n=0$
which can be rewritten as $$\sum_{j=2}^n \phi(g_{k1} f_j-g_{kj} f_1)v_{j-1}=0.$$

Let $\iota \colon R^{n-1} \to R^n$ be the natural embedding to the last $n-1$ components. 
Since $\mathbf{f}' \in \sqrt[\rr]{N'}$, it follows that
\begin{equation}
\label{eq4}
\iota(\mathbf{f}')=\sum_{j=2}^n (g_{k1} f_j-g_{kj} f_1)\mathbf{e}_j=g_{k1} \mathbf{f} -f_1 \mathbf{g}_k \in \iota(\sqrt[\rr]{N'}).
\end{equation}
In the next paragraph, we will show that 
\begin{equation}
\label{eq5}
\iota(\sqrt[\rr]{N'}) \subseteq \sqrt[\rr]{\iota(N')}.
\end{equation}
On the other hand, $\iota(N')$ is contained in $N$ since it is generated by 
\begin{equation}
\label{eq6}
\iota(\mathbf{g}'_i)= \sum_{j=2}^n (g_{k1} g_{ij}- g_{kj}g_{i1})\mathbf{e}_j=g_{k1} \mathbf{g}_i-g_{i1} \mathbf{g}_k \in N
\end{equation}
for $i=1,\ldots,m$. Claim 1 follows from \eqref{eq4}, \eqref{eq5} and \eqref{eq6} since
$$g_{k1} \mathbf{f} =f_1 \mathbf{g}_k+\iota(\mathbf{f}') \in \sqrt[\rr]{N}.$$

Let us show that $N'':=\iota^{-1}(\sqrt[\rr]{\iota(N')})$
is a real submodule of $M'=R^{n-1}$. Suppose that $\sum_i \mathbf{h}_i \otimes \mathbf{h}_i \in 
M' \otimes N''+N'' \otimes M'$, It follows that
$\sum_i \iota(\mathbf{h}_i) \otimes \iota(\mathbf{h}_i) \in \iota(M') \otimes \iota(N'')+\iota(N'') \otimes \iota(M')
\subseteq M \otimes \sqrt[\rr]{N'}+\sqrt[\rr]{N'} \otimes M$. Since $\sqrt[\rr]{N'}$ is real, it follows that
$\iota(\mathbf{h}_i) \in \sqrt[\rr]{N'}$ for all $i$, and so, $\mathbf{h}_i \in N''$ for all $i$.
Since $N''$ is real and it contains $N'$, it also contains $\sqrt[\rr]{N'}$. This proves \eqref{eq5}.

\subsection*{Claim 2} $\mathbf{f} \in \sqrt[\rr]{I \mathbf{f}}$ where $I=\sum_{k=1}^m \sum_{l=1}^n R g_{kl}$. 

\medskip

Since $\mathbf{f} \in \sqrt[\cE]{N}$, it follows that for every $j=1,\ldots,n$ and every $\phi \in V_\RR(R)$,
we have $\phi(f_j)=0$ whenever $\phi(g_{1j})=\ldots=\phi(g_{mj})=0$. (Pick $\mathbf{u}$ from the standard basis of $\RR^n$.) By Theorem \ref{thm1}, there exists
an integer $t_j$ and elements $s_{ij} \in R$ such that $f_j^{2t_j}+\sum_i s_{ij}^2 \in I$.
If we multiply this with $\mathbf{f} \otimes \mathbf{f}\in M \otimes M$, we get that
$$f_j^{t_j} \mathbf{f} \otimes f_j^{t_j} \mathbf{f}+\sum_i s_{ij} \mathbf{f} \otimes s_{ij} \mathbf{f} \in I(\mathbf{f} \otimes \mathbf{f})
\subseteq M \otimes I\mathbf{f}+I\mathbf{f} \otimes M.$$
It follows that $f_j^{t_j} \mathbf{f} \in \sqrt[\rr]{I\mathbf{f}}$ for every $j=1,\ldots,n$.
If $t$ is large enough then $(\sum_{j=1}^n R f_j)^t \subseteq \sum_{j=1}^n R f_j^{t_j}$ which implies that 
\begin{equation}
\label{eq7}
(\sum_{j=1}^n R f_j)^t \mathbf{f} \subseteq (\sum_{j=1}^n R f_j^{t_j})\mathbf{f} \subseteq \sqrt[\rr]{I\mathbf{f}}.
\end{equation}

On the other hand,  $2 (\mathbf{h} \otimes \mathbf{h})=(\sum_{j=1}^n h_j \mathbf{e}_j) \otimes \mathbf{h}+\mathbf{h}\otimes(\sum_{j=1}^n h_j \mathbf{e}_j)
= \sum_{j=1}^n (\mathbf{e}_j \otimes h_j \mathbf{h}+h_j \mathbf{h} \otimes \mathbf{e}_j) \in M \otimes (\sum_{j=1}^n R h_j) \mathbf{h} +(\sum_{j=1}^n R h_j) \mathbf{h} \otimes M$
for every $\mathbf{h} \in M=R^n$, which implies that
$\mathbf{h} \in \sqrt[\rr]{(\sum_{j=1}^n R h_j) \mathbf{h}}.$
By induction, it follows that
\begin{equation}
\label{eq8}
\mathbf{f} \in \sqrt[\rr]{(\sum_{j=1}^n R f_j)^t \mathbf{f}}
\end{equation}
for every $t$. Claim 2 follows from \eqref{eq7} and \eqref{eq8}.
\end{proof}

If $R=\RR[x_1,\ldots,x_d]$ and $g_1,\ldots,g_m,f \in R$ are linear polynomials such that $f(a)=0$ for every $a \in R^d$
satisfying $g_1(a)=\ldots,g_m(a)=0$, then $f$ is clearly a linear combination of $g_1,\ldots,g_m$. This observation does not
generalize to free modules of rank $\ge 2$. Namely, take 
$$\mathbf{g}_1=(x_1,x_1+x_2), \quad \mathbf{g}_2=(-x_1,x_1-x_2), \quad \mathbf{f}=(x_1,0)$$
from $R^2$. Note that $\mathbf{g}_1,\mathbf{g_2}$ and $\mathbf{f}$ are linear and that  $\mathbf{f}$ is not a linear combination of $\mathbf{g}_1$ and $\mathbf{g}_2$. 
On the othe hand,
$$\mathbf{f} \in \sqrt[\rr]{R\mathbf{g}_1+R\mathbf{g}_2} \subseteq \sqrt[\cE]{R\mathbf{g}_1+R\mathbf{g}_2}$$
since 
$\mathbf{f} \otimes \mathbf{f}=\mathbf{r}_1 \otimes \mathbf{g}_1+\mathbf{g}_1 \otimes \mathbf{r}_1+\mathbf{r}_2 \otimes \mathbf{g}_2+\mathbf{g}_2 \otimes \mathbf{r}_2$
where $\mathbf{r}_1 =(\frac{x_1-x_2}{4},0)$ and $ \mathbf{r}_2=(\frac{-x_1-x_2}{4},0)$.

\section{A one-sided Real Nullstellensatz for Matrices}

As the main application of Theorem \ref{thm2} we prove the one-sided Real Nullstellensatz for matrix polynomials
which was already conjectured in \cite[Section 6]{chmn} but it was proved there only in the case of one variable.

Let $n$ be an integer, $R$  a finitely generated commutative unital $\RR$-algebra and
$\chrisA=M_n(R)$ the algebra of all $n \times n$ matrices over $R$ with transpose as involution.
The set $V_\RR(\chrisA)$ consists of all pairs $(\phi,\mathbf{u})$ where $\phi \in V_\RR(R)$ and $\mathbf{u} \in \RR^n$.
Note that is equal to $V_\RR(R^n)$. 

For every left ideal $J$ of $\chrisA$, we define its \textbf{real zero set} 
$$Z_\RR(J):=\{(\phi,\mathbf{u}) \in V_\RR(\chrisA) \mid \phi_n(G)\mathbf{u}=0 \text{ for all } G \in J\}$$
($\phi_n(G)$ means that we apply $\phi$ entrywise to $G$) and its \textbf{real saturation}
$$\sqrt[\cE]{J}:=\{F \in \chrisA \mid \phi_n(F)\mathbf{u}=0 \text{ for all } (\phi,\mathbf{u}) \in Z_\RR(J)\}.$$
We say that a left ideal $J$ of $\chrisA$ is \textbf{real} if for every
$H_1,\ldots,H_k \in \chrisA$ such that $H_1^T H_1+\ldots+H_k^T H_k \in J+J^T$
we have that $H_1,\ldots,H_k \in J$. The smallest real left ideal of $\chrisA$
which contains a given left ideal $J$ of $\chrisA$ is denoted by $\sqrt[\rr]{J}$
and called the \textbf{real radical} of $J$. The analogue of formulas \eqref{eq1} and \eqref{eq3}
can be found in \cite[Section 5]{chmn}.

\begin{thm}
\label{thm3}
For every left ideal $J$ of $M_n(R)$, $$\sqrt[\cE]{J}=\sqrt[\rr]{J}.$$
\end{thm}

\begin{proof}
Let $N$ be the set of all rows of all elements of $J$. Clearly, $N$ is a submodule of $M=R^n$ and
$$J=\left[ \begin{array}{c} N \\ \vdots \\ N \end{array} \right].$$
We claim that
$$\sqrt[\cE]{J}=\left[ \begin{array}{c} \sqrt[\cE]{N} \\ \vdots \\ \sqrt[\cE]{N} \end{array} \right]=
\left[ \begin{array}{c} \sqrt[\rr]{N} \\ \vdots \\ \sqrt[\rr]{N} \end{array} \right]=\sqrt[\rr]{J}.$$
The first equality is clear from the definitions and the second equality follows from Theorem \ref{thm2}.
To prove the third equality, it suffices to show (by a slight abuse of notation) that $J$ is real if and only if $N$ is real.
In the following, we will identify $M \otimes M$ with $M_n(R)$ by sending $\mathbf{a} \otimes \mathbf{b}$
into $\mathbf{a}^T \mathbf{b}$. This identification sends $M \otimes N$ into $J$ and $N \otimes M$ into $J^T$.

Suppose that $J$ is real and $\sum_i \mathbf{m}_i \otimes \mathbf{m}_i \in M \otimes N+N \otimes M$
for some $\mathbf{m}_i \in M$. It follows that $\sum_i \mathbf{m}_i^T \mathbf{m}_i \in J+J^T$. Therefore,
$$\left[ \begin{array}{c} \mathbf{m}_i \\ \vdots \\ 0 \end{array} \right] \in J$$
for every $i$. This proves that $\mathbf{m}_i \in N$ for every $i$, and so, $N$ is real.

Conversely, suppose that $N$ is real and $\sum_i H_i^T H_i \in J+J^T$ for some $H_i \in M_n(R)$.
It follows that $\sum_{i,j} \mathbf{h}_{ij}^T \mathbf{h}_{ij} \in J+J^T$ where
$$H_i = \left[ \begin{array}{c} \mathbf{h}_{i1} \\ \vdots \\ \mathbf{h}_{in} \end{array} \right]$$
for every $i$. It follows that $\sum_{i,j} \mathbf{h}_{ij} \otimes \mathbf{h}_{ij} \in M \otimes N+N \otimes M$.
Since $N$ is real, $\mathbf{h}_{ij} \in N$ for all $i,j$, which shows that $H_i \in J$ for all $i$.
\end{proof}

\end{document}